\documentclass[a4paper,12pt,reqno]{amsart}
\usepackage[utf8]{inputenc}
\usepackage{amsmath, amsthm, amssymb, enumerate, todonotes, cleveref}
\newtheorem{theorem}{Theorem}[section]
\newtheorem*{theorem*}{Theorem}
\newtheorem{lemma}[theorem]{Lemma}

\theoremstyle{definition}

\theoremstyle{remark}
\newtheorem{remark}[theorem]{Remark}
\numberwithin{equation}{section}
\usetikzlibrary{angles,quotes,patterns,arrows.meta,bending,decorations.markings,intersections}
\usetikzlibrary{calc}

\def\a{\mathbf{a}}
\def\c{\mathbf{c}}

\def\b{\mathbf{b}}

\def\etta{\boldsymbol{\eta}}

\def\zetta{\boldsymbol{\zeta}}
\def\llambda{\boldsymbol{\lambda}}

\title{On Askey's extension of  Clausen's Identity and its polynomial perturbation}

\author{Dmitrii Karp$^{1,2}$}
\address{$^{1}$Holon Institute of Technology, Holon, Israel} 
\address{$^{2}$Institute of Mathematics and Informatics, Bulgarian Academy of Sciences, Sofia, Bulgaria} 
\email{dimkrp@gmail.com} 

\author{Vinay Shukla$^{3,4,\#}$}{\thanks{$^{\#}$Corresponding author} }
\address{$^{3}$School of Mathematical Sciences,  Shanghai Jiao Tong University, 800 Dongchuan RD, Shanghai 200240, China} 
\address{$^{4}$Department of Mathematics, School of Computer Science Engineering and Technology, Bennett University, Greater Noida, India} 
\email{vinayshukla4321@gmail.com, vnshukla01@sjtu.edu.cn}

\allowdisplaybreaks

\begin{document}
	
	
\begin{abstract}
 The celebrated Clausen’s identity expresses the square of the Gauss hypergeometric series ${}_2F_{1}(a,b;a+b+1/2;x)$  as a single hypergeometric ${}_3F_2$ series. Goursat showed in 1883 that replacing $1/2$ by $m+1/2$ leads to a hypergeometric series for the square whenever $m$ is a positive integer. Askey found this series explicitly for $m=1$.  The first goal of this paper is to extend this result by treating the case of any natural $m$.  The ${}_3F_{2}$ series on the right-hand side is thereby replaced by its perturbation by an explicit characteristic polynomial of degree 
$2m$, i.e., its coefficients are multiplied by values of this polynomial at nonnegative integers.  The second goal of this paper is to make one further step and replace the square of the Gauss function by its product with its perturbation by an arbitrary polynomial of degree $s\le{2m+1}$.  We show that such product remains hypergeometric and find its explicit form in terms of a polynomial perturbation of the ${}_3F_2$ series. We present an explicit formula for the characteristic polynomial whose degree is shown to be $2m+s$. 
\end{abstract}
\maketitle

\section{Introduction}

A series $\sum_{k}c_k$ is called hypergeometric if $c_{k+1}/c_{k}$ is a rational function of $k$. If  $R(t)=P(t)/Q(t)$ is (another) rational function, we will call the series  $\sum_{k}R(k)c_k$ a {\it rational perturbation} of the original series $\sum_{k}c_k$. Clearly, a rational perturbation of a hypergeometric series remains hypergeometric. When $Q(t)\equiv1$, i.e. $R(t)$ reduces to a polynomial, we will call $\sum_{k}R(k)c_k$ a {\it polynomial perturbation} of the original series $\sum_{k}c_k$. 

\medskip

Write  $(a)_k=\Gamma(a+k)/\Gamma(a)$ for the rising factorial (or Pochhammer's symbol). If $\a=(a_1,\ldots,a_p)$ is a $p$-tuple, denote by  $(\a)_{k}$ the product $(\a)_{k}=(a_1)_{k}\cdots(a_p)_{k}$ and by $\a+\alpha$ the $p$-tuple $(a_1+\alpha,\ldots,a_p+\alpha)$.
Now, if the original series is written in the standard form as a generalized hypergeometric function ${}_pF_q(\a;\b;x)$ \cite[(2.1.2)]{AAR}, then we can also write its rational perturbation using the standard notation as follows:
$$
\sum_{k=0}^{\infty}\frac{(\a)_{k}}{(\b)_kk!}R(k)x^k=R(0){}_{p+d+r}F_{q+d+r}\!\left(\left.\begin{array}{c}\a,1-\llambda,-\etta\\\b,-\llambda,1-\etta\end{array}\:\right\vert x\right),
$$
where  $\llambda=(\lambda_1,\ldots,\lambda_d)$, $\etta=(\eta_1,\ldots,\eta_r)$ are the zeros of $P$ and $Q$, respectively. This formula explains why the function on the right-hand side is frequently referred to as the hypergeometric function with integral parameter differences. The left-hand side, however, only involves the values of the function $R$ at nonnegative integers and does not require the knowledge of its (usually unknown) zeros. Hence, the following notation for the rational/polynomial perturbation turns out to be very useful:
$$
F\!\left(\begin{array}{c}\a\\\b\end{array}\,\bigg\vert\,R\,\bigg\vert\, x\right):=\sum_{k=0}^{\infty}\frac{(\a)_{k}}{(\b)_kk!}R(k)x^k.
$$
In the case of polynomial perturbation, $R(t)=P(t)$ of ${}_pF_{q}$ (or, indeed, of any power series), another way of writing  it is
\begin{equation}\label{eq:xDform}
F\!\left(\begin{matrix}\a\\\b\end{matrix}\:\bigg\vert P\:\bigg\vert x\right)=P(x\partial_{x})
{}_{p}F_{q}\!\left(\left.\begin{matrix}\a\\\b\end{matrix}\:\right\vert x\right),
\end{equation}
where $\partial_{x}$ denotes differentiation with respect to  $x$. This  formula is immediate on noting that  $x\partial_{x}[x^n]=nx^{n}$.  This representation will play a crucial role in Section~3. 

Recent decades have witnessed extensive research activity in the topic of polynomial perturbations of various summation and transformation formulas for hypergeometric functions, frequently referred to as ``extensions by integral parameter differences'' or similarly.  A brief overview can be found in the introduction to the recent paper \cite{karp2025polynomial} by the first author.

One important class of hypergeometric identities is represented by product formulas.
In general, a product of two hypergeometric series is not hypergeometric. Instead, the power series coefficients of a general ${}_pF_{q}\cdot{}_{r}F_{s}$  satisfy a recurrence relation whose order may be as large as\footnote{We thank Professor Yi Zhang for this order bound} $\max(q + 1, p)\max(s + 1, r) \left[ 1 + 2 \left( 1 + \max(p , q, r, s)\right) \max(q + 1, p) \max(s + 1, r) \right]$.  However, there are several known cases when the order of such a recurrence reduces to unity, i.e., the product can be expressed by a single hypergeometric series. The oldest result of this type is Euler's transformation for the Gauss hypergeometric function ${}_2F_1$ \cite[(2.2.7)]{AAR}, which can be viewed as a product formula of the type ${}_1F_{0}\cdot{}_2F_1={}_2F_{1}$.   Another celebrated example is 1828 Clausen's identity given by \cite{Clausen_1828}, \cite[pp.116, 180]{AAR}
\begin{equation}\label{eq:Clausen}
\Bigg[{}_{2}F_{1}\!\left(\begin{matrix}c,d\\c+d+1/2\end{matrix}\,\bigg\vert\,x\right)\Bigg]^2={}_{3}F_{2}\!\left(\begin{matrix}2c,2d,c+d\\c+d+1/2,2c+2d\end{matrix}\,\bigg\vert\,x\right).   
\end{equation}
The next important result in this direction is due to Goursat (1883) \cite{Goursat1883} who showed that the series $[{}_2F_1(a,b;c;x)]^2$ is hypergeometric if and only if $c=a+b+1/2+m$ with $m\in\mathbb{N}\cup\{0\}$. However, he gave no explicit formula. Motivated by a question of the Chudnovsky brothers, who investigated improvements of  Ramanujan's series for $\pi$, Dick Askey in \cite{Askey_1989} found the explicit form for Goursat's series for $m=1$, namely \cite[eq. (3.7)]{Askey_1989}
\begin{equation}\label{eq:Askey3.7}
\left[{}_2F_{1}\left(\begin{array}{c}
 a,b\\ a+b+\frac{3}{2}
\end{array} \middle| x\right) \right]^2 = {}_{5}F_{4}\left(\begin{array}{c}
 2a,2b, a+b, \zeta_1+1, \zeta_2+1\\ a+b+\frac{3}{2}, 2a+2b+2, \zeta_1, \zeta_2
\end{array} \middle| x\right),
\end{equation}
where $\zeta_1$, $\zeta_2$ are negated zeros of the second-degree polynomial
$$
t\to t^2+(4a+4b+1+8ab)t+2(a+b)(2a+1)(2b+1).
$$
Askey derived his result from a summation formula for $4$-balanced very well-poised ${}_7F_6$, using  Whipple's  ${}_7F_6$ transformation \eqref{Wipple Transform} (see also \cite[Theorem 3.4.4]{AAR}). Askey accompanied his derivation with the following comment: ``this method clearly extends to give the sum of the $2m$-balanced very well poised ${}_7F_6$, but the resulting identity is too messy to be worth stating until it is needed.''  
Our first goal in this paper is to give a formula hinted at by Askey, in a form which is not all that messy. 
This is accomplished in the subsequent Section~2, furnishing a hypergeometric expression for the square of ${}_2F_1(a,b;a+b+m+1/2;x)$  for any natural $m$.

Note that one can verify whether a given product of hypergeometric functions is hypergeometric by an application of a combination of Zeiberger's (to find a recurrence for the power series coefficients of the product) and Petkovšek's  (to verify whether the recurrence found in the previous step has a hypergeometric solution) algorithms. Alternatively, one can use algorithms of $D$-finite closure properties to deduce the differential equation satisfied by the product, and then transform it into the recurrence equations for the coefficients. For the square of unrestricted ${}_2F_{1}(a,b;c;x)$, this has been done manually in a recent paper \cite{Mao_Tian_2026} by Mao and Tian. They showed that the power series coefficients of $({}_2F_{1})^2$ satisfy a recurrence of order two and found an explicit form of this recurrence. 

In a recent paper \cite{karp2025polynomial}, the first author found an extension of Clausen's formula \eqref{eq:Clausen} by replacing ${}_2F_1(a,b;a+b+1/2;x)$ on the left-hand side by its rational perturbation of the form ${}F(a,b;a+b+m+1/2|F_m|x)$, where $F_m$ is an arbitrary polynomial of degree $m$. In Section~3 of the present paper, we replace the square by a product of ${}_2F_1(a,b;a+b+m+1/2;x)$ and its perturbation by a polynomial of an arbitrary degree $s$ satisfying $s\le{2m+1}$.

We conclude this introduction by mentioning other relatively recent analogues of Clausen's identity.  Such analogues of importance in number theory were discovered by Chan, Tanigawa, Yang, and Zudilin in \cite{CTYZ2011} using the theory of modular forms. Further investigations along these lines can be found in \cite{ASZ2011}.  In the same year 2011 as the above two references, another type of generalization of Clausen's identity \eqref{eq:Clausen} was published by Vidunas in \cite{Vidunas2011}.  His generalization is in terms of two-variate hypergeometric series, and the bottom parameter of the ${}_2F_1$ function on the left-hand side is unrestricted. 

\section{Contiguous extension of Clausen's identity}

Wipple's transformation \cite[(2.5)]{Askey_1989} between a very well-poised terminating ${}_7F_{6}$ and a $1$-balanced terminating ${}_4F_{3}$ is given by
\begin{align}\label{Wipple Transform}
&{}_7F_{6}\left(\begin{array}{c}
 a,\frac{a}{2}+1,b,c,d,e,-n\\ \frac{a}{2},a+1-b,a+1-c,a+1-d,a+1-e,a+n+1 
\end{array}\right) \nonumber\\
&=\frac{(a+1)_n(a+1-b-c)_n}{(a+1-b)_n(a+1-c)_n}{}_4F_{3}\left(\begin{array}{c}
 -n,a+1-d-e,b,c\\ b+c-n-a,a+1-d,a+1-e 
\end{array}\right),
\end{align}
where from here onward we will omit the argument $1$ from the notation of the hypergeometric function. We will need the following expansion for functions with integral (and positive) parameter differences given by Karlsson  \cite[Eq. (1)-(3)]{Karlsson_1971}
\begin{align}\label{Karlsson formula}
{}_{s+r}F_{t+r}\left(\begin{array}{c}
 \mathbf{a}, \mathbf{f+m}\\ \mathbf{c}, \mathbf{f}
\end{array} \middle|\, x\right) = \sum_{j_1=0}^{m_1} \cdots \sum_{j_r=0}^{m_r} A(\mathbf{j}) x^{|\mathbf{j}|} {}_{s}F_{t}\left(\begin{array}{c}
 \mathbf{a+|\mathbf{j}|} \\ \mathbf{c+|\mathbf{j}|} 
\end{array} \middle|\, x\right),
\end{align}
where $\mathbf{m}=(m_1,\ldots,m_r)$ comprises positive integers, $\mathbf{j}=(j_1,\ldots,j_r)$,  $|\mathbf{j}|=j_1+j_2+\cdots+j_r$ and $A(\textbf{j})$ is given by 
\begin{align*}
A(\textbf{j})= \binom{\mathbf{m}}{\mathbf{j}}\frac{(f_2+m_2)_{j_1} (f_3+m_3)_{j_1+j_2} \ldots (f_r+m_r)_{j_1+\ldots+j_{r-1}} (\textbf{a})_{|\mathbf{j}|}} { (f_1)_{j_1} (f_2)_{j_1+j_2} \ldots (f_r)_{|\mathbf{j}|} (\textbf{c})_{|\mathbf{j}|}}.
\end{align*}
Here binomial of a vector denotes the product of the binomials of the form $\binom{m_k}{j_k}$ and $(\mathbf{a})_{q}$ is the product of the rising factorials $(a_{k})_{q}$.

Our first result is the following theorem giving an extension of Askey's formula \eqref{eq:Askey3.7}, see also \cite[(3.7)]{Askey_1989}. 
\begin{theorem}\label{thrm:extended askey}
Suppose $m\ge0$ is an integer, and no bottom parameter equals a non-positive integer. Then, the following product formula holds:
\begin{equation}\label{eq:extendedAskey}
\left[{}_2F_{1}\left(\begin{array}{c}
 a,b\\ a+b+m+\frac{1}{2}
\end{array} \middle|\,x\,\right) \right]^2 = F\!\left(\!\!\begin{array}{c}
 2a,2b, a+b\\ a+b+m+\frac{1}{2}, 2a+2b+2m
\end{array}  \middle| \,P_{2m}^{a,b} \,\middle|\,x\right),
\end{equation}
where the polynomial $P_{2m}^{a,b}(t)$ of degree $2m$ is given by
\begin{equation}\label{eq:Pab}
P_{2m}^{a,b}(t)\!=\!\sum_{j=0}^{m}\!\binom{m}{j} \frac{(a)_m (b)_m (a+b+\frac{t}{2}+j)_{m-j}(-\frac{t}{2})_j}{(-1)^{j}(a+b)_{m}(a+\frac{1}{2})_{m}  (b+\frac{1}{2})_{m}}{}_3F_{2}\!\left(\!\!\begin{array}{c}
 -m+j,\frac{1-t}{2}-a-b-m,1/2\\ 1-a-m, 1-b-m
\end{array}\!\right).
\end{equation}
 \end{theorem}

 \begin{remark}
If  $\boldsymbol{\lambda}= (\lambda_1,\lambda_2,\ldots,\lambda_{2m})$ are negated zeros of $P_{2m}^{a,b}(t)$, then we can rewrite \eqref{eq:extendedAskey} as follows
\begin{equation}\label{Generalized Clausen Identity}
\left[{}_2F_{1}\!\left(\!\begin{array}{c}
 a,b\\ a+b+m+\frac{1}{2}
\end{array} \middle|\,x\right) \right]^2\!=  {}_{3+2m}F_{2+2m}\left(\begin{array}{c}
 2a,2b, a+b, \boldsymbol{\lambda}+1\\ a+b+m+\frac{1}{2}, 2a+2b+2m, \boldsymbol{\lambda}
\end{array} \middle| x\right).
\end{equation}
 \end{remark}

\begin{proof}
Setting $e=2a+n-m-b-c-d+1$ in \eqref{Wipple Transform}, we will get ${}_4F_{3}$ on the right-hand side containing a bottom parameter  $f:=b+c-n-a$  and a top parameter $f+m$. Hence, we are in the position to apply \eqref{Karlsson formula} with $r=1$, $s=3$, $t=2$ and $f=b+c-n-a$ at $x=1$, yielding 
\begin{multline}\label{After using Karlsson formula}
\frac{(a+1)_n(a+1-b-c)_n}{(a+1-b)_n(a+1-c)_n}{}_4F_{3}\left(\begin{array}{c}
 -n,b+m+c-n-a,b,c\\ b+c-n-a,a+1-d,b+m+c+d-n-a
\end{array}\right)\\
= \sum_{j=0}^{m} A(j) {}_3F_{2}\left(\begin{array}{c}
 -n+j,b+j,c+j\\ a+1-d+j,b+m+c+d-n-a+j
\end{array} \right),
\end{multline}
where
\begin{align*}
 A(j) = \binom{m}{j} \frac{(-n)_j (b)_j (c)_j}{(a+1-d)_j (b+m+c+d-n-a)_j (b+c-n-a)_j}  . 
\end{align*}
The function ${}_3F_{2}$ on the right hand side is $m-j+1$ balanced. It can be transformed into a sum of $m-j+1$ terms by an application of \cite[Appendix, formula (VIII)]{Srinivasa_Jeugt_Raynal_Jagannathan_Rajeswari_1992} which reads
\begin{align*}
{}_3F_{2} \!\left(\!\!\begin{array}{c}
 A,B,-N\\ D,E
\end{array}\!\right)\!=\! (-1)^{N} \frac{(D-A)_N (D-B)_N}{(D)_N (E)_N} {}_3F_{2} \!\left(\!\!\begin{array}{c}
 1-S,1-D-N,-N\\ 1+A-D-N,1+B-D-N
\end{array}\!\right) ,
\end{align*}
where $S=D+E+N-A-B$. Setting $N = n-j$, $A=b+j$, $B=c+j$, $D=a+1-d+j$ and $E=b+m+c+d-n-a+j$, in the above formula leads to 
\begin{multline}\label{After applying Rao formula}
LHS~ of ~ \eqref{After using Karlsson formula} =  \sum_{j=0}^{m} A(j) \frac{(-1)^{n-j} (a+1-d-b)_{n-j} (a+1-d-c)_{n-j}}{(a+1-d+j)_{n-j} (b+c+d+m-n-a+j)_{n-j}} \\ 
\times{}_3F_{2}\left(\begin{array}{c}
 -m+j,d-a-n,-n+j\\ d-n+b+j-a,d-a-n+c+j
\end{array}\right).
\end{multline}
Next, we put $c=a/2$ and $d=(a+1)/2$ on the LHS of \eqref{Wipple Transform} and similarly in \eqref{After applying Rao formula}. Now, the balancing condition becomes $e=a+n-m-b+1/2$, and we obtain
\begin{align*}
&{}_4F_{3}\!\left(\!\!\begin{array}{c}
 a,b,e,-n\\ a+1-b,a+1-e,a+n+1 
\end{array}\!\right)=\frac{(a+1)_n(a/2+1-b)_n}{(a+1-b)_n(a/2+1)_n} \sum_{j=0}^{m} A(j)  \nonumber \\
&\times\frac{((a+1)/2-b)_{n-j} (1/2)_{n-j}}{((a+1)/2+j)_{n-j} (1/2-b-m)_{n-j}} {}_3F_{2}\!\left(\!\!\begin{array}{c}
 -m+j,\frac{1-a}{2}-n,-n+j\\ \frac{1-a}{2}-n+b+j,\frac{1}{2}-n+j
\end{array}\!\right),
\end{align*}
where
\begin{align*}
 A(j) = \binom{m}{j} \frac{(-n)_j (b)_j (a/2)_j}{((a+1)/2)_j (b+m-n+1/2)_j (b-n-a/2)_j}  . 
\end{align*}
Replace $a$ by $-k$, $k\in\mathbb{N}$, in the above formula and change $-n$ to (a new) $a$ to get
\begin{align}\label{Second replace}
&{}_4F_{3}\!\left(\!\!\begin{array}{c}
 -k,b,e,a\\ -k+1-b,-k+1-e,-k-a+1 
\end{array}\!\right)=\frac{(-k+1)_{-a}(-k/2+1-b)_{-a}}{(-k+1-b)_{-a}(-k/2+1)_{-a}} \sum_{j=0}^{m} A(j)   \nonumber \\
&\frac{((-k+1)/2-b)_{-a-j} (1/2)_{-a-j}}{((-k+1)/2+j)_{-a-j} (1/2-b-m)_{-a-j}} {}_3F_{2}\!\left(\!\!\begin{array}{c}
 -m+j,(k+1)/2+a,a+j\\ (k+1)/2+a+b+j,1/2+a+j
\end{array}\!\right),
\end{align}
where for any $w$ we used the notation $(z)_w=\Gamma(z+w)/\Gamma(z)$
and the balancing condition becomes  $e=1/2-k-a-b-m$. The coefficients $A_j$ now take the form 
\begin{align*}
 A(j)=\binom{m}{j} \frac{(a)_j (b)_j (-k/2)_j}{(-k/2+1/2)_j (b+m+a+1/2)_j (b+a+k/2)_j}   .
\end{align*}
To prove that \eqref{Second replace} is indeed a true equality, it suffices to show that both sides are rational functions of $a$. Indeed, as they  coincide for an infinite number of values $a=0,-1,-2\ldots$, they must be identical. 
Next, we will establish the rationality of the right hand side of \eqref{Second replace} in the variable $a$.

To simplify the coefficients at ${}_3F_2$ above apply Legendre's duplication formula, the recurrence relation for the Gamma function, and the identities $(x)_n=(-1)^n (1-x-n)_n$ and $(x)_{p+q}=(x)_p (x+p)_q$. This leads to 
\begin{align}\label{Simplification}
&\frac{(1-k)_{-a}(1-b-k/2)_{-a} (1/2-b-k/2)_{-a} (1/2)_{-a}}{(1-k-b)_{-a}(1-k/2)_{-a}(1/2-k/2)_{-a}(1/2-b-m)_{-a}} \nonumber \\
=&\frac{(1-k)_{-a}(1-k-2b)_{-2a} (1/2)_{-a}}{(1-k-b)_{-a}(-k+1)_{-2a} (1/2-b-m)_{-a}} \nonumber \\
=&\frac{\Gamma(1-k-2b-2a)\Gamma(1/2-a)\Gamma(1-k-b)\Gamma(1-k-a)\Gamma(1/2-b-m)}{\Gamma(1-k-2b)\Gamma(1/2)\Gamma(1-k-b-a)\Gamma(1-k-2a)\Gamma(1/2-b-m-a)} \nonumber \\
=&\frac{\Gamma(1/2-a)\Gamma(1-k-a)\Gamma(1/2-b)\Gamma(1-k-b)}{\Gamma(1-k-2b)\Gamma(1/2)\Gamma(1-k-2a) } \nonumber \\
& \times \frac{ \Gamma(1-k-2b-2a-2m)(1-k-b-a)_{-m}}{\Gamma(1/2-b-m-a)\Gamma(1-k-b-a-m)(1-k-2b-2a)_{-2m} (1/2-b-m)_m}.
\end{align}
Further, in view of 
\begin{align*}
&\Gamma(1/2-A) \Gamma(1-A-k)= 2^{2A} \sqrt{\pi}\frac{\Gamma(1-2A) }{(A)_k}, \\
&\frac{\Gamma(1-2A)}{\Gamma(1-k-2A)} = \frac{\Gamma(2A+k)}{\Gamma(2A)}=(2A)_k,
\end{align*}
expression in \eqref{Simplification} takes the form
\begin{align*}
&RHS~of~\eqref{Simplification}=\frac{2^{2a} \sqrt{\pi} 2^{2b} \sqrt{\pi}}{2^{2a+2b+2m} \sqrt{\pi} \sqrt{\pi}} \frac{(1-k-b-a)_{-m}}{(1-k-2b-2a)_{-2m} (\frac{1}{2}-b-m)_m} \nonumber\\
&\times \frac{\Gamma(1-2a) }{\Gamma(1-k-2a)(a)_k}\frac{\Gamma(1-2b) }{\Gamma(1-k-2b)(b)_k} \frac{\Gamma(1-k-2b-2a-2m)(a+b+m)_k }{\Gamma(1-2a-2b-2m)} \nonumber\\
& = \frac{1}{2^{2m}} \frac{(1-k-b-a)_{-m}}{(1-k-2b-2a)_{-2m} (\frac{1}{2}-b-m)_m} \frac{(2a)_k (2b)_k (a+b+m)_k}{(a)_k(b)_k (2a+2b+2m)_k}.
\end{align*}
Substituting the above into \eqref{Second replace}, we obtain (with $e=1/2-k-a-b-m$)
\begin{multline}\label{4F3 to M}
{}_4F_{3}\left(\begin{array}{c}
 -k,b,e,a\\ -k+1-b,-k+1-e,-k-a+1 
\end{array}\right) \\= \frac{(k+2b+2a)_{2m}}{ 2^{2m}(k+b+a)_{m}  (b+1/2)_m} \frac{(2a)_k (2b)_k (a+b+m)_k}{(a)_k(b)_k (2a+2b+2m)_k} P(k),
\end{multline}
where $P(k)$ is a polynomial given by
\begin{multline*}
P(k) =\sum_{j=0}^{m} \binom{m}{j} \frac{(-1)^{-j}  (a)_j (b)_j (-k/2)_j}{(a+b+k/2)_j ((k+1)/2+a+b)_{j} (a+1/2)_{j}} 
\\
\times{}_3F_{2}\left(\begin{array}{c}
 -m+j,(k+1)/2+a,a+j\\ (k+1)/2+a+b+j,1/2+a+j
\end{array}\right).    
\end{multline*}
This expression is manifestly rational in $a$ proving the claim that the right hand side of \eqref{Second replace} is rational in $a$.  Further, by the Cauchy product, we have
\begin{align*}
\left[{}_2F_{1}\left(\begin{array}{c}
 a,b\\ c
\end{array} \middle| x\right) \right]^2 &= \sum_{k=0}^{\infty} x^k \sum_{n=0}^{k} \frac{(a)_n (b)_n (a)_{k-n}(b)_{k-n}}{(c)_k (c)_{k-n} n! (k-n)!} \nonumber\\
& =\sum_{k=0}^{\infty} \frac{(a)_k (b)_k}{(c)_k k!} {}_4F_{3}\left(\begin{array}{c}
 -k,a,b,1-k-c\\ 1-k-a,1-k-b,c
\end{array}\right) x^k .
\end{align*}
Substituting $c=a+b+m+\frac{1}{2}$ here, we get
\begin{multline}\label{After putting c}
\left[{}_2F_{1}\left(\begin{array}{c}
 a,b\\ a+b+m+\frac{1}{2}
\end{array} \middle| x\right) \right]^2 \\= \sum_{k=0}^{\infty} \frac{(a)_k (b)_k}{(a+b+m+\frac{1}{2})_k k!} {}_4F_{3}\left(\begin{array}{c}
 -k,a,b,\frac{1}{2}-k-a-b-m\\ 1-k-a,1-k-b,a+b+m+\frac{1}{2}
\end{array}\right)x^k .
\end{multline}
The ${}_4F_3$ here matches that in \eqref{4F3 to M} due to the balancing condition  $e=1/2-k-a-b-m$ as given below  \eqref{Second replace}.  Hence, summing ${}_4F_3$ in \eqref{After putting c} by \eqref{4F3 to M}  we obtain
\begin{multline}\label{Before using Rao VI}
 \left[{}_2F_{1}\left(\begin{array}{c} 
 a,b\\ a+b+m+\frac{1}{2}
\end{array} \middle| x\right) \right]^2
\\
=\sum_{k=0}^{\infty}  \frac{(2a)_{k}(2b)_{k}(a+b+m)_{k}(2a+2b+k)_{2m}}{2^{2m}(a+b+m+\frac{1}{2})_{k} (2a+2b+2m)_{k}(a+b+k)_{m} (\frac{1}{2}+b)_{m}k!} P(k)x^k
\\
= \sum_{k=0}^{\infty}  \frac{(2a)_k (2b)_k (a+b)_k (2a+2b)_{2m}}{2^{2m}(a+b+m+\frac{1}{2})_k (2a+2b)_k (a+b)_{m}  (\frac{1}{2}+b)_m k!} P(k) x^k,
\end{multline}
where we have used the formulas
\begin{align*}
& \frac{(a+b+m)_k }{(a+b+k)_m} = \frac{(a+b)_k }{(a+b)_m}, \quad \frac{(2a+2b+k)_{2m} }{(2a+2b+2m)_k} = \frac{(2a+2b)_{2m} }{(2a+2b)_k}.
\end{align*}
Next, an application of \cite[Appendix (VI)]{Srinivasa_Jeugt_Raynal_Jagannathan_Rajeswari_1992} reading 
\begin{align*}
{}_3F_{2} \left(\begin{array}{c}
 A,B,-N\\ D,E
\end{array}\right) = (-1)^{N} \frac{(1-S)_N (B)_N}{(D)_N (E)_N} {}_3F_{2} \left(\begin{array}{c}
 E-B,D-B,-N\\ 1-B-N,S-N
\end{array}\right),
\end{align*}
where $S=D+E+N-A-B$, to ${}_3F_{2}$ in the definition of  $P(k)$, we get
\begin{align*}
&P(k) =\sum_{j=0}^{m} \binom{m}{j} \frac{(-1)^{j}  (a)_j (b)_j (-\frac{k}{2})_j}{(a+b+k/2)_j (a+\frac{1}{2})_{j} (\frac{k+1}{2}+a+b)_{j} }  \\
&\hspace{2cm}
\times\frac{ (\frac{1}{2}+b+j)_{m-j} (a+j)_{m-j}}{(\frac{1+k}{2}+a+b+j)_{m-j} (\frac{1}{2}+a+j)_{m-j}}
{}_3F_{2}\left(\!\!\begin{array}{c}
 -m+j,\frac{1+k}{2}+b,1/2\\ 1-a-m, \frac{1}{2}+b+j
\end{array}\right) \\
& =\!\sum_{j=0}^{m} \binom{m}{j} \frac{(-1)^{j}(a)_{m}(b)_{j}(\frac{1}{2}+b+j)_{m-j}(-k/2)_j}{(a+\frac{1}{2})_{m}(a+b+k/2)_j (a+b+k/2+1/2)_{m}} {}_3F_{2}\left(\!\!\begin{array}{c}
 -m+j,\frac{1+k}{2}+b,1/2\\ 1-a-m, \frac{1}{2}+b+j
\end{array}\right).
\end{align*}
Writing 
\begin{align*}
\frac{1}{(a+b+k/2)_j} = \frac{(a+b+k/2+j)_{m-j}}{(a+b+k/2)_m}  ,   
\end{align*}
and factoring out $(a+b+k/2)_m (a+b+k/2+1/2)_{m}=2^{-2m}(2a+2b+k)_{2m}$ from the inner sum, the RHS of \eqref{Before using Rao VI} can be written as 
 \begin{align}\label{After factoring out denominator k s}
\sum_{k=0}^{\infty}\frac{(2a)_k (2b)_k (a+b)_k (2b+2a)_{2m}}{(2a+2b+k)_{2m}(a+b+m+\frac{1}{2})_k (2a+2b)_k k!} P^{a,b}_{2m}(k) x^k,
\end{align}
where 
\begin{multline}\label{Mk Final}
P^{a,b}_{2m}(k)= \sum_{j=0}^{m} \binom{m}{j} \frac{(-1)^{j} (a)_{m} (b)_j (\frac{1}{2}+b+j)_{m-j}(a+b+k/2+j)_{m-j}(-k/2)_j}{(a+b)_{m}  (b+1/2)_m(a+1/2)_{m} }\\ \times{}_3F_{2}\left(\begin{array}{c}
 -m+j,\frac{1+k}{2}+b,1/2\\ 1-a-m, \frac{1}{2}+b+j
\end{array}\right) 
\end{multline}
is related to $P(k)$ by
$$
P^{a,b}_{2m}(k)=\frac{P(k)}{(a+b)_{m}(b+1/2)_m}.
$$ 
To make $P^{a,b}_{2m}$ look symmetric in  $a$, $b$, apply \cite[Appendix (V)] {Srinivasa_Jeugt_Raynal_Jagannathan_Rajeswari_1992}  reading 
\begin{align*}
{}_3F_{2} \left(\begin{array}{c}
 A,B,-N\\ D,E
\end{array}\right) = \frac{(D-B)_N }{(D)_N } {}_3F_{2} \left(\begin{array}{c}
 E-A,B,-N\\ 1+B-D-N,E
\end{array}\right)
\end{align*}
 to obtain
\begin{multline*}
P^{a,b}_{2m}(k) = \sum_{j=0}^{m} \binom{m}{j} \frac{(-1)^{j} (a)_{m} (b)_j (\frac{1}{2}+b+j)_{m-j} (a+b+\frac{k}{2}+j)_{m-j}(-\frac{k}{2})_j}{(a+b)_{m}(b+\frac{1}{2})_m(a+\frac{1}{2})_{m}}
\\
\times\frac{(b+j)_{m-j}}{(\frac{1}{2}+b+j)_{m-j}}{}_3F_{2}\left(\begin{array}{c}
 -m+j,\frac{1-k}{2}-a-b-m,1/2\\ 1-a-m, 1-b-m
\end{array}\right),
\end{multline*}
which simplifies to \eqref{eq:Pab}.

Finally, we return to \eqref{After factoring out denominator k s} and apply the transformation
\begin{align*}
\frac{(2b+2a)_{2m}}{(2b+2a)_{k}(2a+2b+k)_{2m}}=\frac{1}{(2b+2a+2m)_{k}}
\end{align*}
to get \eqref{eq:extendedAskey}.
\end{proof}

\begin{remark}
For $m=0$, formula \eqref{Generalized Clausen Identity} becomes the classical Clausen's identity \eqref{eq:Clausen}.
\end{remark}

\begin{remark}
The summation formula implied by equating coefficients in \eqref{eq:extendedAskey} and given in \eqref{4F3 to M} can be simplified to  
\begin{align*}
{}_4F_{3}\left(\begin{array}{c}
 -k,a,b,\frac{1}{2}-k-a-b-m\\ 1-k-a,1-k-b,a+b+m+\frac{1}{2}
\end{array}\right) = \frac{(2a)_{k} (2b)_{k} (a+b)_{k} }{(a)_k (b)_k(2a+2b+2m)_{k}}P_{2m}^{a,b}(k).
\end{align*}
Isolating $P_{2m}^{a,b}(k)$ from this formula, we can record another expression for  $P_{2m}^{a,b}$ using Lagrange or Newton interpolation in the form \cite[Lemma~1]{karp2025polynomial}
\begin{equation}\label{eq:LagrangeP}
			P_{n}(t)=\frac{(-1)^n}{n!}\sum\limits_{k=0}^{n}\binom{n}{k}(-t)_k(t-n)_{n-k}a_k=\sum\limits_{k=0}^{n}\frac{(-t)_{k}}{k!}\sum\limits_{j=0}^{k}(-1)^j\binom{k}{j}a_{j}
\end{equation}
for the polynomial of degree $n$ satisfying $P_n(j)=a_j$ for $j=0,\ldots,n$.  For example, the first formula yields:
\begin{multline*}
P_{2m}^{a,b}(t)=\frac{1}{(2m)!}\sum\limits_{k=0}^{2m}\binom{2m}{k}\frac{(-t)_k(t-2m)_{2m-k}(a)_k (b)_k(2a+2b+2m)_{k}}{(2a)_{k} (2b)_{k} (a+b)_{k}}
\\
\times {}_4F_{3}\left(\begin{array}{c}
 -k,a,b,\frac{1}{2}-k-a-b-m\\ 1-k-a,1-k-b,a+b+m+\frac{1}{2}
\end{array}\right).
\end{multline*}
\end{remark}

\begin{remark} 
The recurrence for the power series coefficients of the square $\big({}_2F_{1}\big)^2$ obtained recently in \cite[Theorem~2.1]{Mao_Tian_2026} yields, in particular, a recurrence for the values of $P_{2m}^{a,b}(t)$ at positive integers $t=n\in\mathbb{N}$:     
\begin{align}\label{eq:recurrence Pab}
P^{a,b}_{2m}(n+1) =  \beta_0(n) P^{a,b}_{2m}(n)+\beta_1(n)P^{a,b}_{2m}(n-1),   
\end{align}
where $c=a+b+m+1/2$ and
\begin{align*}
& \beta_0(n)=\frac{2n^3+3(a+b+c-1)n^2+((a+b)(4c-3)+4ab-c+1)n+2ab(2c-1)}{(2a+n) (2b+n) (a+b+n)}, \\
& \beta_1(n)=-\frac{n(c+n-1)(2c-2+n)}{(2a+n) (2b+n) (a+b+n) }.
\end{align*}
\end{remark}

\textbf{Example~1}.
For $m=1$ we recover Askey's formula \eqref{eq:Askey3.7} found in \cite[eq. (3.7)]{Askey_1989}. Indeed, formula \eqref{Generalized Clausen Identity} assumes the form
\begin{align*}
\left[{}_2F_{1}\left(\begin{array}{c}
 a,b\\ a+b+\frac{3}{2}
\end{array} \middle| x\right) \right]^2 = {}_{5}F_{4}\left(\begin{array}{c}
 2a,2b, a+b, 1+\zeta_1, 1+\zeta_2\\ a+b+\frac{3}{2}, 2a+2b+2, \zeta_1, \zeta_2
\end{array} \middle| x\right),
\end{align*}
 where $\zeta_1$ and $\zeta_2$ are the zeros of the characteristic polynomial obtained by putting $m=1$ in \eqref{eq:Pab}.  A simple calculation yields
\begin{equation*}
P_{2}^{a,b}(-t)=\frac{t^2-(4a+4b+1+8ab)t+2(a+b)(2a+1)(2b+1)}{2(a+b)(2a+1)(2b+1)}.
\end{equation*}
Up to normalization, this is precisely the characteristic polynomial from  \eqref{eq:Askey3.7} obtained in \cite[eq. (3.8)]{Askey_1989}.

\medskip

\textbf{Example~2}.
For $m=2$ we get
\begin{align*}
\left[{}_2F_{1}\left(\begin{array}{c}
 a,b\\ a+b+\frac{5}{2}
\end{array} \middle| x\right) \right]^2 = {}_{7}F_{6}\left(\begin{array}{c}
 2a,2b, a+b, 1+\zetta\\ a+b+\frac{5}{2}, 2a+2b+4, \zetta
\end{array} \middle| x\right),
\end{align*}
where $\zetta=(\zeta_1,\zeta_2,\zeta_3,\zeta_4)$ are negated roots of the characteristic polynomial obtained by putting $m=2$ in \eqref{eq:Pab}. Namely,
\begin{multline*}
 64(a+1/2)_{2}(b+1/2)_{2}(a+b)_{2}P_{4}^{a,b}(t)\!
 =3t^4 +2(8ab+12(a+b)+9)t^3
 \\
 +
\big(64a^2b^2+4(a+b)(36ab+27)+72(a^2+b^2)+288ab+33\big)t^2 
\\
+\big((a+b)(64a^2b^2+360ab+66)
+4(a^2+b^2)(32ab+27)+48(a^3+b^3)+288a^2b^2+26ab+9\big)2t
\\
+64(a+1/2)_{2}(b+1/2)_{2}(a+b)_{2}.
\end{multline*}

\section{Polynomial perturbation of the extended Askey's formula}

We will use the standard notation $S(n,k)$ for  Stirling's numbers of the second kind generated by
$$
x^n=\sum\limits_{k=0}^{n}S(n,k)[x]_{k},~~\text{where}~[x]_{k}=x(x-1)\cdots(x-k+1),
$$
and computed by means of the relation
$$
S(n,k)=\frac{1}{k!}\sum\limits_{i=0}^{k}(-1)^{k-i}\binom{k}{i}i^n.
$$
The main goal of this section is to prove the following theorem.  
\begin{theorem}\label{th:extAskeyPerturbed}
Suppose $s\le2m+1$ and  $F_s(y)=\sum_{n=0}^{s}\sigma_{n}y^n$. Then the following product formula holds:
\begin{multline}\label{eq:extAskeyPerturbed}
{}_{2}F_{1}\!\left(\!\!\begin{array}{l}a,b\\a+b+m+\frac{1}{2}\end{array}\vline\:x\!\right)F\!\left(\!\!\begin{array}{l}a,b\\a+b+m+\frac{1}{2}\end{array} \!\middle| \,F_s\,\middle|\,x\!\right)
\\
=F\!\left(\!\!\begin{array}{c}
 2a,2b, a+b\\ a+b+m+1/2, 2a+2b+2m
\end{array} \!\middle| \,\hat{P}_{2m+s}\,\middle|\,x\right),
\end{multline}
where the characteristic polynomial $\hat{P}_{2m+s}$ of degree $2m+s$ is given by any of the following two expressions
\begin{multline}\label{eq:hatP-interpolation}
\hat{P}_{2m+s}(t)=\frac{(-1)^s}{(2m+s)!}\sum\limits_{k=0}^{2m+s}\binom{2m+s}{k}(-t)_k(t-2m-s)_{2m+s-k}
\\
\times\frac{(a)_{k} (b)_{k}(2a+2b+2m)_{k}}{(2a)_{k} (2b)_{k}(a+b)_{k}}F\!\left(\!\!\begin{array}{c} -k, a,b ,\frac{1}{2}-k-a-b-m \\ a+b+m+\frac{1}{2}, 1-a-k,1-b-k\end{array}\middle|\,F_s \,\right)
\end{multline}
and
\begin{multline}\label{eq:hatP}
\hat{P}_{2m+s}(t)=\sigma_0P_{2m}^{a,b}(t)
+\frac{1}{2}\sum_{n=1}^{s}\sigma_n\sum_{k=1}^{n}(-1)^{k}S(n,k)\sum\limits_{j=0}^{\lfloor k/2\rfloor}\frac{k}{k-j}\binom{k-j}{j}
\\
\times\frac{(a)_{j}(b)_{j}(a+b+m)_{j}(1/2-a-b-m-t)_{j}}{4^{j}(a+1/2)_{j}(b+1/2)_{j}(a+b)_{2j}}
(-t)_{k}P_{2m-2j}^{a+j,b+j}(t-2j),
\end{multline}
with $P_{2q}^{\alpha,\beta}$ defined in \eqref{eq:Pab}.
\end{theorem}
\begin{remark}
By exchanging the order of summation in   \eqref{eq:hatP-interpolation} we can also write  $\hat{P}_{2m+s}$  in the form
\begin{multline*}
\hat{P}_{2m+s}(t)=\frac{(1-t)_{2m+s}}{(2m+s)!}\sum\limits_{j=0}^{2m+s}\frac{(-2m-s)_{j}(a)_{j}(b)_{j}(2a+2b+2m)_{j}(-t)_{j}F_s(j)}{(2a)_{j}(2b)_{j}(a+b)_{j}j!(1-t)_{j}}
\\
\times
{}_{6}F_{5}\!\left(\!\!\begin{array}{l}-2m-s+j,a,b,a+b+m+j+1/2,2a+2b+2m+j,-t+j\\a2+j,2b+j,a+b+j,a+b+m+1/2,1-t+j\end{array}\!\right).
\end{multline*}
This form makes dependence on $F_s$ more explicit at the price of making the dependence on $t$ more complicated.
\end{remark}

The proof of this theorem will rely on three lemmas. The first is straightforward.
\begin{lemma}
For each integer $n\ge0$, we have
\begin{equation}\label{eq:DFperturbed}
  \frac{d^n}{dx^n}F\left(\!\!\begin{array}{c}
 \mathbf{a}\\ \mathbf{b}
\end{array} \middle| \,P \,\middle|\,x\right)
=\frac{(\mathbf{a})_{n}}{(\mathbf{b})_{n}}
F\left(\!\!\begin{array}{c}
 \mathbf{a}+n\\ \mathbf{b}+n
\end{array} \middle| \,S_nP \,\middle|\,x\right),
\end{equation}
where $S_n$ is the shift operator, $S_nP(x)=P(x+n)$.
\end{lemma}
Next, we need a kind of orthogonality relation given by
\begin{lemma}
For each integer $k\ge1$, the following identity holds
\begin{equation}\label{eq:binomialidentity}
\sum\limits_{j=0}^{\lfloor k/2\rfloor}\frac{(-1)^{j}k}{k-j}\binom{k-j}{j}\binom{k-2j}{\ell-j}=\delta_{\ell,0}+\delta_{\ell,k}=\begin{cases}1,~\ell=0,k\\0, ~\ell=1,2\ldots, k-1\end{cases}.
\end{equation}
\end{lemma}
\begin{proof}
 Let
$$
H_{\ell,k}:=\sum_{j=0}^{\lfloor k/2\rfloor}\frac{(-1)^{j}k}{k-j}\binom{k-j}{j}\binom{k-2j}{\ell-j}.
$$
The required relation is then equivalent to
\begin{equation}\label{eq:generatig_form}
H(x):=\sum_{\ell=0}^{k} H_{\ell,k} x^\ell=1+x^k. 
\end{equation}
In view of the binomial formula and the fact that $\binom{m}{n}=0$ for $n\notin\{0,\ldots,m\}$ we have
$$
\sum_{\ell=0}^{k}\binom{k-2j}{\ell-j}x^{\ell}=x^{j}(1+x)^{k-2j}. 
$$
Exchanging the order of summation in the definition of $H(x)$ then yields
$$
H(x)=\sum_{j=0}^{\lfloor k/2\rfloor}\frac{(-1)^{j}k}{k-j}\binom{k-j}{j}x^j(1+x)^{k-2j}
=(1+x)^k\sum_{j=0}^{\lfloor k/2\rfloor}\frac{(-1)^{j}k}{k-j}\binom{k-j}{j}\bigg[\frac{x}{(1+x)^2}\bigg]^j.
$$
In order to establish \eqref{eq:generatig_form}, we will show that 
\begin{equation}\label{eq:S-Tk}
H(x)=2x^{k/2}T_k\Big(\frac{1+x}{2\sqrt{x}}\Big),
\end{equation}
where $T_k$ stands for the Chebyshev polynomial of the first kind. 
Indeed, one of the known expressions for $T_k$ is \cite[(1.108)]{Rivlin1990}
$$
T_{k}(y)=\frac{k}{2}\sum_{j=0}^{\lfloor k/2\rfloor}(-1)^{j} \frac{(k-j-1)!}{j!(k-2j)!}(2y)^{k-2j}=\frac{1}{2}(2y)^{k}\sum_{j=0}^{\lfloor k/2\rfloor} \frac{(-1)^{j}k}{k-j}\binom{k-j}{j}\bigg(\frac{1}{4y^2}\bigg)^{j}.
$$
Setting
$$
\frac{x}{(1+x)^2}=\frac{1}{4y^2}~\Leftrightarrow~y=\frac{1+x}{2\sqrt{x}}~\Leftrightarrow~(2y)^{k}=\frac{(1+x)^k}{x^{k/2}},
$$
we arrive at \eqref{eq:S-Tk}. From the definition $T_{k}(\cos(\theta))=\cos(k\theta)$, it follows that
$$
T_k\bigg(\frac{t+t^{-1}}{2}\bigg)=\frac{t^k+t^{-k}}{2},
$$
so that putting $t=\sqrt{x}$ implies
$$
\frac{t+t^{-1}}{2}=\frac{\sqrt{x}+1/\sqrt{x}}{2}=\frac{1+x}{2\sqrt{x}}.
$$ 
Hence, we get
$$
T_k\Big(\frac{1+x}{2\sqrt{x}}\Big)=\frac{x^{k/2}+x^{-k/2}}{2}.
$$
Substituting this into \eqref{eq:S-Tk}, we arrive at \eqref{eq:generatig_form}, thus also proving \eqref{eq:binomialidentity}.  
\end{proof}

The key ingredient of the proof of Theorem~ \ref{th:extAskeyPerturbed} is the following
\begin{lemma}
For a smooth function $f$ and any integer $n\ge1$, we have
\begin{equation}\label{eq:f-theta-f}
f(x)(x\partial_{x})^{n}f(x)= \frac{1}{2}\sum_{k=0}^{n}S(n,k)x^k\sum\limits_{j=0}^{\lfloor k/2\rfloor}\frac{(-1)^{j}k}{k-j}\binom{k-j}{j}\partial_x^{k-2j}[f^{(j)}(x)]^2.
\end{equation}
\end{lemma}
\begin{proof} First, we will prove the formula 
\begin{equation}\label{eq:-theta-f-theta}
2f(x)\partial_x^kf(x)=\sum\limits_{j=0}^{\lfloor k/2\rfloor}\frac{(-1)^{j}k}{k-j}\binom{k-j}{j}\partial_x^{k-2j}[(\partial_x^{j}f(x)]^2.
\end{equation}
Treating $[(\partial_x^{j}f(x)]^2$ as product, by Leibniz rule we have:
\begin{multline*}
\partial_x^{k-2j}[(\partial_x^{j}f(x)]^2=
\sum_{r=0}^{k-2j}\binom{k-2j}{r}\partial_x^{r+j}f(x)\partial_x^{k-j-r}f(x)
\\
=\sum_{\ell=j}^{k-j}\binom{k-2j}{\ell-j}\partial_x^{\ell}f(x)\partial_x^{k-\ell}f(x)= \sum_{\ell=0}^{k}\binom{k-2j}{\ell-j}\partial_x^{\ell}f(x)\partial_x^{k-\ell}f(x),
\end{multline*}
where we put $\ell=j+r$ for the second equality, and for the third equality we used the fact that $\binom{m}{n}=0$ when $n<0$ or $n>m$. Substituting this into the right-hand side of \eqref{eq:-theta-f-theta} and changing the order of summations yields by \eqref{eq:binomialidentity}
$$
\sum_{\ell=0}^{k}\partial_x^{\ell}f(x)\partial_x^{k-\ell}f(x)\sum\limits_{j=0}^{\lfloor k/2\rfloor}\frac{(-1)^{j}k}{k-j}\binom{k-j}{j}\binom{k-2j}{\ell-j}=2f(x)\partial_x^kf(x)
$$
which is \eqref{eq:-theta-f-theta}.  

We are now ready to prove \eqref{eq:f-theta-f}.     To this end, we need the well-known expansion 
$$
(x\partial_{x})^{n}=\sum_{k=0}^{n}S(n,k)x^k\partial_{x}^{k}.
$$
Applying it to $f$, multiplying the result by $f$, and using \eqref{eq:-theta-f-theta}, we arrive at \eqref{eq:f-theta-f}.
\end{proof}

To illustrate \eqref{eq:f-theta-f}, below are some simple particular cases.

\medskip

\textbf{Examples.} Note that $S(0,0)=1$, $S(n,0)=0$ for $n\ge1$. Then we have 
$$
f(x)(x\partial_{x})f(x)=\frac{1}{2}x\partial_{x}[f(x)]^2,
$$
$$
f(x)(x\partial_{x})^2f(x)=\frac{1}{2}x\partial_{x}[f(x)]^2+\frac{1}{2}x^2\partial_{x}^2[f(x)]^2-x^2[f'(x)]^2,
$$
$$
f(x)(x\partial_{x})^3f(x)=\frac{1}{2}x\partial_{x}[f(x)]^2+\frac{3}{2}x^2\partial_{x}^2[f(x)]^2-3x^2[f'(x)]^2+\frac{1}{2}x^3\partial_{x}^3[f(x)]^2-\frac{3}{2}x^3\partial_{x}[f'(x)]^2.
$$

\medskip

\textit{Proof of Theorem~\ref{th:extAskeyPerturbed}}.
Write 
\begin{equation*}
    f(x)={}_2F_{1}\left(\begin{array}{c}
 a,b\\ a+b+m+\frac{1}{2}
\end{array} \middle| x\right),
\end{equation*}
so that 
$$
    f^{(j)}(x)=\frac{(a)_{j}(b)_{j}}{(a+b+m+1/2)_{j}} {}_2F_{1}\left(\begin{array}{c}
 a+j,b+j\\ a+b+m+\frac{1}{2}+j
\end{array} \middle| x\right).
$$
Hence, by \eqref{eq:extendedAskey} with $a\to  a+j$, $b\to b+j$, $m\to m-j$, we obtain
\begin{multline}\label{eq:fjsquare}
[f^{(j)}(x)]^2=\frac{(a)_{j}^2(b)_{j}^2}{(a+b+m+1/2)_{j}^2}
\\\times F\left(\begin{array}{c}
 2a+2j,2b+2j, a+b+2j\\ a+b+m+j+\frac{1}{2}, 2a+2b+2m+2j
\end{array}  \middle| \,P_{2m-2j}^{a+j,b+j} \,\middle|\,x\right).
\end{multline}
Denote
$$
\alpha=\frac{(a)_{j}^2(b)_{j}^2}{(a+b+m+1/2)_{j}^2}.
$$
Using this notation, \eqref{eq:fjsquare} implies by employing \eqref{eq:DFperturbed}, that
\begin{multline*}
    \partial_{x}^{k-2j}[f^{(j)}(x)]^2={\alpha}\partial_{x}^{k-2j}
 F\!\left(\!\!\begin{array}{c}
 2a+2j,2b+2j, a+b+2j\\ a+b+m+j+\frac{1}{2}, 2a+2b+2m+2j
\end{array}  \middle| \,P_{2m-2j}^{a+j,b+j} \,\middle|\,x\right)
\\
=\alpha\frac{(2a+2j)_{k-2j}(2b+2j)_{k-2j}(a+b+2j)_{k-2j}}{(a+b+m+j+1/2)_{k-2j}( 2a+2b+2m+2j)_{k-2j}}
\\
 \times F\!\left(\!\!\begin{array}{c}
 2a+k,2b+k, a+b+k\\ a+b+m+k-j+1/2, 2a+2b+2m+k
\end{array}  \middle| \,S_{k-2j}P_{2m-2j}^{a+j,b+j}\,\middle|\,x\right).
\end{multline*}
Simplifying the Pochhammer symbols and denoting 
$$
Q_{k,j}(t)=S_{k-2j}P_{2m-2j}^{a+j,b+j}(t)=P_{2m-2j}^{a+j,b+j}(t+k-2j),
$$
which is a polynomial of degree $2m-2j$, we obtain
\begin{equation}\label{eq:Dfjsquared}
\partial_{x}^{k-2j}[f^{(j)}(x)]^2=
    A_{j}B_{k}C_{k,j}
     F\!\left(\!\!\begin{array}{c}
 2a+k,2b+k, a+b+k\\ a+b+m+1/2+k-j, 2a+2b+2m+k
\end{array} \!\middle| \,Q_{k,j}\,\middle|\,x\right), 
\end{equation}
where
$$
A_j\!=\!\frac{4^{-j}(a)_{j}(b)_{j}(a+b+m)_{j}}{(a+1/2)_{j}(b+1/2)_{j}(a+b)_{2j}},~B_k\!=\!\frac{(2a)_{k}(2b)_{k}(a+b)_{k}}{(2a+2b+2m)_{k}},~~~C_{k,j}\!=\!\frac{1}{(a+b+m+1/2)_{k-j}}.
$$
Next, we have 
$$
F\!\left(\!\!\begin{array}{l}a,b\\a+b+m+\frac{1}{2}\end{array} \!\middle| \,F_s\,\middle|\,x\!\right)
=\sum_{n=0}^{s}\sigma_n (x\partial_{x})^{n} \bigg[{_{2}F_{1}}\!\left(\!\!\begin{array}{l}a,b\\c\end{array}\vline\:x\!\right)\bigg],
$$
where $\sigma_n$ are the coefficients of $F_s(y)$, $F_s(y)=\sum_{n=0}^{s}\sigma_{n}y^n$. Multiplying this formula by ${}_2F_1(a,b;c;x)$ with $c=a+b+m+1/2$ replacing each product by expansion \eqref{eq:f-theta-f} and using \eqref{eq:Dfjsquared} we get 
\begin{multline*}
{_{2}F_{1}}\!\left(\!\!\begin{array}{l}a,b\\c\end{array}\vline\:x\!\right)F\!\left(\!\!\begin{array}{l}a,b\\a+b+m+\frac{1}{2}\end{array} \!\middle| \,F_s\,\middle|\,x\!\right)=
\\
=\bigg[{_{2}F_{1}}\!\left(\!\!\begin{array}{l}a,b\\c\end{array}\vline\:x\!\right)\bigg]^2+\sum_{n=1}^{s}\sigma_{n}\cdot {}_{2}F_{1}\!\left(\!\!\begin{array}{l}a,b\\c\end{array}\vline\:x\!\right)(x\partial_{x})^{n} \bigg[{_{2}F_{1}}\!\left(\!\!\begin{array}{l}a,b\\c\end{array}\vline\:x\!\right)\bigg]
\\
= F\!\left(\!\!\begin{array}{c}
 2a,2b, a+b\\ a+b+m+1/2, 2a+2b+2m
\end{array} \!\middle| \,P_{2m}^{a,b}\,\middle|\,x\right)+\frac{1}{2}\sum_{n=1}^{s}\sigma_n\sum_{k=1}^{n}S(n,k)B_kx^k
\\
\times\sum\limits_{j=0}^{\lfloor k/2\rfloor}\frac{(-1)^{j}k}{k-j}\binom{k-j}{j}A_{j}C_{k,j}F\!\left(\!\!\begin{array}{c}
 2a+k,2b+k, a+b+k\\ a+b+m+1/2+k-j, 2a+2b+2m+k
\end{array} \!\middle| \,Q_{k,j}\,\middle|\,x\right).
\end{multline*}
Simplifying we obtain 
\begin{multline*}
x^kB_kC_{k,j}F\!\left(\!\!\begin{array}{c}
 2a+k,2b+k, a+b+k\\ a+b+m+1/2+k-j, 2a+2b+2m+k
\end{array} \!\middle| \,Q_{k,j}\,\middle|\,x\right)
\\
=\sum\limits_{l=0}^{\infty}\frac{(2a)_{l}(2b)_l(a+b)_{l}x^{l}(l-k+1)_{k}(a+b+m+1/2+l-j)_{j}}{(a+b+m+1/2)_{l}(2a+2b+2m)_{l}l!}Q_{k,j}(l-k).
\end{multline*}
Note that we extended the range of summation from $l\ge{k}$ to $l\ge0$ in view of the fact that $(l-k+1)_{k}=0$ for $l=0,1,\ldots,k-1$. Substituting this into the above product formula, we arrive at 
$$
{_{2}F_{1}}\!\left(\!\!\begin{array}{l}a,b\\c\end{array}\vline\:x\!\right)F\!\left(\!\!\begin{array}{l}a,b\\a+b+m+\frac{1}{2}\end{array} \!\middle| \,F_s\,\middle|\,x\!\right)=
F\!\left(\!\!\begin{array}{c}
 2a,2b, a+b\\ a+b+m+1/2, 2a+2b+2m
\end{array} \!\middle| \,\hat{P}_{2m+s}\,\middle|\,x\right),
$$
where the characteristic polynomial is given by
\begin{multline*}
\hat{P}_{2m+s}(t)=P_{2m}^{a,b}(t)+\frac{1}{2}\sum_{n=1}^{s}\sigma_n\sum_{k=1}^{n}S(n,k)
\\
\times\sum\limits_{j=0}^{\lfloor k/2\rfloor}\frac{(-1)^{j}k}{k-j}\binom{k-j}{j}A_{j}
(t-k+1)_{k}(a+b+m+1/2+t-j)_{j}Q_{k,j}(t-k)
\\
=P_{2m}^{a,b}(t)+\frac{1}{2}\sum_{n=1}^{s}\sigma_n
\sum_{k=1}^{n}(-1)^{k}S(n,k)\sum\limits_{j=0}^{\lfloor k/2\rfloor}\frac{k}{k-j}\binom{k-j}{j}
\\
\times\frac{(a)_{j}(b)_{j}(a+b+m)_{j}(1/2-a-b-m-t)_{j}}{4^{j}(a+1/2)_{j}(b+1/2)_{j}(a+b)_{2j}}
(-t)_{k}P_{2m-2j}^{a+j,b+j}(t-2j).
\end{multline*}
This shows that $\hat{P}_{2m+s}(t)$ is indeed a polynomial of degree $2m+s$ given by \eqref{eq:hatP}. 

To establish \eqref{eq:hatP-interpolation}, we will apply the interpolation polynomial as given in \eqref{eq:LagrangeP}.
To this end we  can write the coefficients of the  product \eqref{eq:extAskeyPerturbed} in the form:
\begin{multline*}
F\left(\begin{array}{c}
 a,b\\ a+b+m+\frac{1}{2}
\end{array} \middle|\,F_s\middle|\, x\right) {}_2F_{1}\left(\begin{array}{c}
 a,b\\ a+b+m+\frac{1}{2}
\end{array} \middle|\,x\right) \\
= \sum_{l=0}^{\infty} \frac{(a)_{l}(b)_{l}x^{l}}{(a+b+m+\frac{1}{2})_{l}l!} F\!\left(\!\! \begin{array}{c} -l, a,b ,\frac{1}{2}-l-a-b-m \\ a+b+m+\frac{1}{2}, 1-a-l,1-b-l\end{array} \middle|\,F_s\right),
\end{multline*}
so that 
$$
 \hat{P}_{2m+s}(l)=\frac{(a)_l (b)_l(2a+2b+2m)_{l}}{(2a)_{l} (2b)_{l}(a+b)_{l}}
F\!\left(\!\! \begin{array}{c} -l, a,b ,\frac{1}{2}-l-a-b-m \\ a+b+m+\frac{1}{2}, 1-a-l,1-b-l\end{array} \middle|\,F_s\right),
$$ 
and \eqref{eq:LagrangeP}  immediately yields
\eqref{eq:hatP-interpolation}. $\hfill\square$

\medskip

\begin{remark}
The product formula \eqref{eq:extAskeyPerturbed} implies a summation formula by the standard argument of equating power series coefficients on both sides, namely
\begin{equation}\label{eq:extAskeyPeturbed-summation}
F\left(\begin{array}{c} -k, a,b ,\frac{1}{2}-k-a-b-m \\ a+b+m+\frac{1}{2}, 1-a-k,1-b-k\end{array} \middle| F_s \right) = \frac{(2a)_{k} (2b)_{k} (a+b)_{k} \hat{P}_{2m+s}(k)}{(a)_{k} (b)_{k}(2a+2b+2m)_{k}}.
\end{equation}
\end{remark}

\begin{remark}
The condition $s\le2m+1$ cannot be removed: formula \eqref{eq:extAskeyPerturbed} becomes invalid when it is violated. At the same time, the polynomial given by \eqref{eq:hatP-interpolation} is still well defined and when used in  \eqref{eq:extAskeyPerturbed}  it yields the first $2m+s$ correct power series coefficients by construction, but not more.  It seems plausible to assume that the right hand side of \eqref{eq:extAskeyPerturbed} retains its shape as a polynomial perturbation of the corresponding ${}_3F_2$,
but the perturbing polynomial must then be of degree higher than $2m+s$. This leads to the following

\smallskip

\noindent\textbf{Open problem.} What should replace the right hand side of \eqref{eq:extAskeyPerturbed} when $s>2m+1$? 
\end{remark}

Finally, we present some explicit particular cases of the formula \eqref{eq:extAskeyPerturbed}. 

\textbf{Example~1}. 
For $\mathbf{s}=1$, we get for any integer $m\ge0$:
\begin{multline*}
{}_2F_{1}\left(\begin{array}{c}
 a,b\\ a+b+m+\frac{1}{2}
\end{array} \middle|\,x\right){}_3F_{2}\left(\begin{array}{c}
 a,b,f+1\\ a+b+m+\frac{1}{2},f
\end{array} \middle|\,x\right)
\\
= F\left(\begin{array}{c}
 2a,2b, a+b,2f+1\\ a+b+m+\frac{1}{2}, 2a+2b+2m,2f
\end{array}  \middle| \,P_{2m}^{a,b}\,\middle|\,x\right),
\end{multline*}
where $P_{2m}^{a,b}$ is defined in \eqref{eq:Pab}.
In particular, setting $m=0$ we obtain
\begin{equation*}
{}_2F_{1}\left(\begin{array}{c}
 a,b\\ a+b+\frac{1}{2}
\end{array} \middle|\,x\right){}_3F_{2}\left(\begin{array}{c}
 a,b,f+1\\ a+b+\frac{1}{2},f
\end{array} \middle|\,x\right)
\\
= {}_{4}F_{3}\left(\begin{array}{c}
 2a,2b, a+b,2f+1\\ a+b+\frac{1}{2}, 2a+2b,2f
\end{array}  \middle|\,x\right),
\end{equation*}
which appears to be a new explicit product identity. 

If we set $f=a$ we get a product formula for two Gauss hypergeometric functions ${}_2F_1$ which also appears to be new 
\begin{multline*}
{}_2F_{1}\left(\begin{array}{c}
 a,b\\ a+b+m+\frac{1}{2}
\end{array} \middle|\,x\right){}_2F_{1}\left(\begin{array}{c}
 a+1,b\\ a+b+m+\frac{1}{2}
\end{array} \middle|\,x\right)
\\
= F\left(\begin{array}{c}
  2a+1, 2b, a+b\\a+b+m+\frac{1}{2}, 2a+2b+2m
\end{array}  \middle| \,P^{a,b}_{2m}\,  \middle|\,x\,\right).
\end{multline*}
In particular, for $m=0$ we have:
\begin{equation*}
{}_2F_{1}\left(\begin{array}{c}
 a,b\\ a+b+\frac{1}{2}
\end{array} \middle|\,x\right){}_2F_{1}\left(\begin{array}{c}
 a+1,b\\ a+b+\frac{1}{2}
\end{array} \middle|\,x\right)\!
= {}_3F_{2}\left(\begin{array}{c}
  2a+1, 2b, a+b\\a+b+\frac{1}{2}, 2a+2b
\end{array}\, \middle|\, x\right),
\end{equation*}
which is equivalent to Orr's formula  \cite[(15)]{Grinshpan_2013}. 
By equating coefficients, the above product identity implies the summation formula slightly different from \eqref{eq:extAskeyPeturbed-summation}, which we record for later use: 
\begin{equation}\label{eq:new4F3sum}
{}_4F_{3}\left(\!\!\begin{array}{c}
 -k,a+1,b,\frac{1}{2}-k-a-b-m\\a+b+m+\frac{1}{2}, 1-k-a,1-k-b
\end{array}\!\!\right)\!=\!\frac{(2a+1)_{k}(2b)_{k}(a+b)_{k}P^{a,b}_{2m}(k)}{(a+b+m+1/2)_{k}(2a+2b+2m)_{k}}.
\end{equation}

\medskip

\textbf{Example~2}. For $\mathbf{s}=2$, we get for $m\ge1$
\begin{multline*}
{}_2F_{1}\left(\begin{array}{c}
 a,b\\ a+b+m+\frac{1}{2}
\end{array} \middle| x\right){}_3F_{2}\left(\begin{array}{c}
 a,b,f+2\\ a+b+m+\frac{1}{2},f
\end{array} \middle| x\right)
\\
= F\left(\begin{array}{c}
 2a,2b, a+b\\ a+b+m+\frac{1}{2}, 2a+2b+2m
\end{array}  \middle|  \hat{P}_{2m+2}  \middle|  x\right),
\end{multline*}
where
\begin{multline*}
 \hat{P}_{2m+2}(t)=\Big(1+\frac{(2f+1)}{2f(f+1)}t+\frac{t^2}{2f(f+1)}\Big)P_{2m}^{a,b}(t)
 \\
 - \frac{ab(a+b+m)(a+b+m-1/2+t)t(t-1)}{f(f+1)(2a+1)(2b+1)(a+b)_2}P_{2m-2}^{a+1,b+1}(t-2).
\end{multline*}
In particular, for $m=1$, we can write using \eqref{eq:hatP-interpolation}:
\begin{multline*}
\hat{P}_{4}(t)=\frac{1}{24}\sum\limits_{k=0}^{4}\binom{4}{k}(-t)_k(t-4)_{4-k}
\\
\times\frac{(a)_{k} (b)_{k}(2a+2b+2)_{k}}{(2a)_{k} (2b)_{k}(a+b)_{k}}\cdot{}_5F_{4}\!\left(\!\!\begin{array}{c} -k, a,b ,\frac{1}{2}-k-a-b-1, f+2\\ a+b+\frac{3}{2}, 1-a-k,1-b-k,f\end{array}\right).
\end{multline*}

\bibliographystyle{siam} 
\bibliography{references.bib}
\end{document}